\documentclass[12pt]{amsart}
\usepackage{verbatim,amscd,amssymb}
\usepackage{graphicx,tcolorbox}

\usepackage[active]{srcltx}

\newtheorem{thm}{Theorem}[section]
\newtheorem{prop}[thm]{Proposition}
\newtheorem{lem}[thm]{Lemma}

\newtheorem{cor}[thm]{Corollary}

\newtheorem*{thmain}{Main Theorem}

\theoremstyle{definition}
\newtheorem{dfn}[thm]{Definition}
\newtheorem{rem}[thm]{Remark}

\def\C{\mathbb{C}}   

\def\R{\mathbb{R}}

\def\0{\emptyset}

\def\Z{\mathbb{Z}}
\renewcommand\emptyset{\varnothing}
\newcommand{\sm}{\setminus}

\def\eps{\varepsilon}
\def\ol{\overline}

\def\si{\sigma}  \def\ta{\theta}  \def\Ga{\Gamma}
\def\al{\alpha}  \def\be{\beta}  \def\la{\lambda} \def\ga{\gamma}

\def\vp{\varphi}
\def\le{\leqslant}
\def\ge{\geqslant}

\def\imp{\mathrm{Imp}}

\def\uc{\mathbb{S}^1}
\def\thu{\mathrm{Th}}
\def\bd{\mathrm{Bd}}

\def\fix{\mathrm{Fix}}

\newcommand{\ind}{\mathrm{Ind}}

\newcommand{\hp}{\widehat P}
\def\disk{\mathbb{D}}
\def\cdisk{\overline{\mathbb{D}}}

\newcommand{\diam}{\mathrm{diam}}

\begin{document}

\title[Cutpoints of subcontinua of polynomial Julia sets]
{Cutpoints of invariant subcontinua of polynomial Julia sets}

\dedicatory{Dedicated to Misha Lyubich's 60-th birthday}

\author[A.~Blokh]{Alexander~Blokh}

\author[L.~Oversteegen]{Lex Oversteegen}

\thanks{The second named author was partially  supported
by NSF grant DMS-1807558}

\author[V.~Timorin]{Vladlen~Timorin}

\thanks{The third named author has been supported by the HSE University Basic Research Program and Russian Academic Excellence Project '5-100'.}

\address[Alexander~Blokh, Lex~Oversteegen]
{Department of Mathematics\\ University of Alabama at Birmingham\\
Birmingham, AL 35294-1170}

\address[Vladlen~Timorin]
{Faculty of Mathematics\\
HSE University, Russian Federation\\
6 Usacheva St., 119048 Moscow
}

\address[Vladlen~Timorin]
{Independent University of Moscow\\
Bolshoy Vlasyevskiy Pereulok 11, 119002 Moscow, Russia}

\email[Alexander~Blokh]{ablokh@math.uab.edu}
\email[Lex~Oversteegen]{overstee@uab.edu}
\email[Vladlen~Timorin]{vtimorin@hse.ru}



\begin{abstract}
We prove  fixed point results for branched covering maps
$f$ of the plane. For complex polynomials $P$ with Julia set $J_P$ these
imply that periodic cutpoints of some invariant subcontinua of $J_P$
are also cutpoints of $J_P$.
We deduce that, under certain assumptions on invariant subcontinua $Q$ of $J_P$,
every Riemann ray to $Q$ landing at a periodic repelling/parabolic point $x\in Q$ is isotopic to a Riemann ray to $J_P$
relative to $Q$.
\end{abstract}

\maketitle

\section{Introduction}\label{s:intro}

The plane fixed point problem is a central problem in continuum theory. While
solving it in full generality is still elusive, advances have been
recently made in certain cases. 
In particular, by \cite{bfmot13}, for \emph{positively oriented branched covering  maps of
the plane} the existence of a fixed point \emph{with specific
properties} can be established even inside non-invariant continua
(provided certain conditions hold). In this paper we use tools from
\cite{bfmot13}, prove new fixed point results in the same spirit, and rely upon them to tackle some
topological problems of polynomial dynamics. In the Introduction we assume knowledge of complex dynamics and give
preliminary versions of our main results; later we make more detailed
statements.

\subsection*{Standing notation throughout the paper}
Write $\R_+$ for the set $\{x\in \R, x>0\}$ of positive reals, $\C$
for the plane of complex numbers, $\widehat\C$ for the Riemann sphere,
$\disk=\{z\in\C\,:\, |z|<1\}$ for the open unit disk centered at the
origin, $\cdisk=\{z\in\C\,:\, |z|\le 1\}$ for the corresponding closed
disk (more generally, $\ol A$ will refer to the closure of $A\subset\C$), and $\uc=\{z\in\C\,:\,|z|=1\}$ for the unit circle.
Fix a
polynomial $P$ of degree $\deg(P)>1$ with connected Julia set $J_P=J$
and filled Julia set $K_P=K$ (if it does not cause ambiguity, we do not
refer to $P$ in our notation). We always assume that the term
$z^{\deg(P)}$ has coefficient $1$ in $P(z)$, where $\deg(P)$ is the
degree of $P$ (this can be achieved by a coordinate change of the form
$z\mapsto \la z$). The boundary of a set $E$ is denoted by $\bd(E)$.
Given a compact
set $Q\subset\C$, denote by $U_\infty(Q)$ the unbounded complementary
domain of $Q$ and by $\thu(Q)$ the set $\C\sm U_\infty(Q)$ called the
\emph{topological hull} of $Q$.
Given a continuum $Q$, consider the topological disk $\widehat\C\sm\thu(Q)$ and
 a conformal isomorphism between this disk and $\widehat\C\sm\cdisk$.
We will always assume that the conformal isomorphism takes $\infty$ to $\infty$ and has real positive derivative at $\infty$.
We refer to such an isomorphism as the \emph{Riemann map} for $\widehat\C\sm\thu(Q)$.
The images under the Riemann map of straight radial rays connecting $\uc$ with $\infty$ are called \emph{$Q$-rays}.

Assume that $K$ is connected.
A periodic repelling/parabolic point $x$ of $P$ is said to be \emph{regular};
a point eventually mapped to a regular periodic point is called a \emph{regular (pre)periodic} point.

The main theorem of this paper relies on a fixed point result which
partially strengthens a theorem of \cite{GM}. Consider the union $\Sigma$ of all
invariant $P$-rays and the set $\fix$ of their landing points. According to \cite{GM},
every complementary component $A$ of $\Sigma$ contains a unique
invariant rotational object (either a rotational fixed point of $P$ or
an invariant parabolic domain). We prove Theorem \ref{t:1a-new}, in which
the component $A$ is replaced by a (not even
necessarily invariant) continuum in $A\cup \fix$ that is locally invariant near
$\bd(A)$ (this makes Theorem \ref{t:1a-new} slightly stronger
that the quoted result of \cite{GM} as far as the \emph{existence} of an
invariant rotational object is concerned).

Let $Q\subset J$ be a continuum.
A point $x\in Q$ is a \emph{cutpoint of $Q$ of order $n$} if $Q\sm \{x\}$ has exactly $n$ components.
Suppose that either $Q$ is full, or $Q=\bd(\thu(Q))$.

\begin{thmain}
Let $x\in K$ be a regular periodic point
and all $K$-rays to $x$ form $m$ wedges $W_i$, where $1\le i\le m$.
Moreover, suppose that $x\in Q$ is a cutpoint of order $n$ of an invariant continuum $Q\subset J$.
Then $n\le m$, each wedge $W_i$ intersects $Q$ over a connected (possibly empty) set,
 and every $Q$-ray to $x$ is isotopic rel. $Q$ to a $K$-ray that lands at $x$.
\end{thmain}

Observe that a point may be a cutpoint of a subcontinuum while not
being a cutpoint of a big continuum. E.g., if a continuum $X$ is formed
by the graph of $\sin(\frac{1}{x})$ on the segment $(0,
\pi]$ and segment $I$
connecting points $(0, -1)$ and $(0, 1)$, then any non-endpoint $x$ of
$I$ is a cutpoint of $I$ of order 2, but \emph{not} a cutpoint of $X$.
The main theorem implies that such behavior is impossible for invariant continua of polynomial Julia sets.

\section{Rotational fixed points in non-invariant continua}\label{s:fxpt-in-cont}

\subsection{Weakly repelling fixed points}
Let $f:X\to X$ be a continuous map of a locally compact metric space $(X, d)$ to itself.
Call an $f$-fixed point $x$
\emph{weakly repelling in the sense of the metric $d$} if, for an open
neighborhood $U$ of $x$, the restriction $f|_U$ is a homeomorphism,
and $d(f(y), x)>d(y, x)$ for any $y\in U, y\ne x$.
Now, let $X$ be a
locally compact  topological space, $x\in X$ be a point, and $f:X\to
X$ be a continuous map. The point $x$ is \emph{weakly repelling} if it
is weakly repelling in the sense of
some metric on $X$ that induces the given topology.
The orbit of any point $y\in U, y\ne x$ escapes any compact subset of $U$.
Indeed, otherwise the sequence of distances $d(x,f^n(y))$ is
increasing and bounded for some $y\ne x$. This implies that for any
limit point $z$ of the orbit of $y$ we have, by continuity,
$d(x,z)=d(x, f(z))$, a contradiction.
Let us call this property the \emph{escaping property at $x$} (in $U$).
Thus, there are no fixed points of $f$ in $U$ except for $x$.

If $\psi:\R_+\to \C$ is an embedding, call
$T=\psi(\R_+)$ a \emph{(topological) ray}. If also $\lim_{t\to
\infty}\psi(t)=\infty$, say that $T$ is \emph{a ray from infinity}. If
$Q$ is a full continuum or the boundary of a full continuum, $Q$-rays
are rays from infinity. If $\lim_{t\to 0} \psi(t)=y\notin T$,
say that $T$ \emph{lands} at $y$. A ray $R$ is
\emph{($f$-)invariant} if $f(R)\supset R$; if $R$ is a ray
from infinity, call $R$ \emph{($f$-)invariant} if $f(R)=R$.
Consider a positively oriented branched covering map $f:\C\to \C$ and an $f$-invariant
ray $T$ landing at a fixed point $y$.
If $y$ is weakly repelling, then, for some $r$ and any $s\in (0, r)$, the point $\psi(s)$ is mapped to $\psi(s')$ with $s'>s$.
(In this case, we say that points of $T$ move \emph{away from} $y$.)
Otherwise the absence of fixed points of $f$ close to $x$ implies that for all small $s$ the point $\psi(s)$
is mapped to $\psi(s')$ with $s'<s$ which contradicts the escaping property at $y$.

\begin{dfn}\label{d:rotat}
Suppose that $f:\C\to \C$ is a positively oriented branched covering map, $X\subset \C$ is
a full continuum, and $x\in X$ is an $f$-fixed point such that, for some
neighborhood $U$ of $x$, we have that $f|_U$ is a homeomorphism and
$f(U\cap X)\subset X$. Then $x$ is called \emph{non-rotational
for $X$} if there exists an invariant topological ray $T\subset \C\sm
X$ that lands at $x$ such that points of $T$ sufficiently close to $x$
move away from $x$.
If such ray $T$ does not exist, then a fixed point $x$ is called \emph{rotational for $X$}.
\end{dfn}

Examples of  rotational fixed points are points contained in the
interior of $X$, or points in the boundary of $X$ but  not accessible,
or  accessible points $x$ at which no invariant topological ray lands
such that its points close to $x$ move away from $x$ along the ray
(\emph{accessible} always means accessible from $U_\infty(X)$).

We do not talk of (non-)rotational fixed points $x$ in the absence of
the continuum $X$ as then the definitions are too inclusive: if $d$ is a metric on $\C$ defining the usual topology,
if $x$ is a fixed point at which $f:\C\to \C$ is a homeomorphism, and if,
for some $\la>1$, we have $d(x, f(y))>\la d(x, y)$, then we can always
construct an invariant topological ray landing at $x$. Indeed, choose a
small open disk $D_0$ around $x$ such that a pullback $D_1$ of $D_0$ is
compactly contained in $D_0$. Choose a topological arc $I_1$ in $\ol
D_0\sm D_1$ connecting a point $y_1\in\bd(D_1)$ with
$y_0=f(y_1)\in\bd(D_0)$ but otherwise lying in $D_0\sm \ol D_1$.
Consider a pullback $I_2$ of $I_1$ connecting $y_1$ with $y_2$, a
pullback of $I_2$ connecting $y_2$ with $y_3$, etc. The countable union
of these iterated pullbacks of $I_1$ is an invariant topological ray landing at
$x$.

The Riemann map  $\vp:\widehat\C\sm K_P\to \widehat\C\sm\cdisk$
conjugates $P|_{\C\sm K_P}$ with the restriction of $z\mapsto z^{\deg(P)}$ to $\C\sm \cdisk$.

\begin{lem}\label{l:invar}
A fixed point $x\in K_P$ is non-rotational for $K_P$ if and only if there
exists an invariant $K_P$-ray landing at $x$ (hence $x$ is regular).
Thus, a fixed point $x\in K$ is rotational if it is either non-regular (i.e.,
attracting, Cremer or Siegel) or regular with non-zero combinatorial
rotation number at $x$.
\end{lem}

\begin{proof}
If there is an invariant $K_P$-ray landing at $x$, then $x$ is
non-rotational for $K_P$. If now $x$ is non-rotational for $K_P$, then
$x\in J=\bd(K_P)$ is accessible from $\C\sm K_P$, and an invariant
\emph{topological} ray $R\subset \C\sm K_P$ lands at $x$.
Let $\psi:\widehat\C\sm \cdisk \to \widehat\C\sm K_P$ be the Riemann map.
The \emph{topological} ray $\psi^{-1}(R)$ lands at a point
$e^{2\pi i\ta}\in\uc$. Then $d\ta=\ta$ modulo 1, and the $K_P$-ray of
argument $\ta$ lands at $x$.
\end{proof}

\subsection{Plane continua and fixed points}
By a \emph{continuum} we mean a compact and connected metric space and
by a \emph{full continuum} $X$ in the plane $\C$ we mean a continuum
$X\subset \C$ such that $\C\sm X$ is connected. A famous theorem by
Brouwer \cite{bro11} states that every continuous map from a disk to
itself has a fixed point (this property is called the \emph{fixed point
property}). This result motivated the following long-standing problem
in topology.

\medskip

\noindent\textbf{Plane Fixed Point Problem \cite{ste35}.}
Does a continuous map of a full plane continuum to itself always have a fixed point?

\medskip

Even though this problem is not solved yet, certain progress has been made.
By \cite{bfmot13}, full plane continua do have the fixed point property for a
restricted class of maps: all \emph{positively oriented branched covering
maps of the plane} (see Definition~2.3.3 in \cite{bfmot13}). Evidently, every map
$f:[a,b]\to \mathbb R$ so that $f(a)>a$ and $f(b)<b$ has a fixed point
(even though the interval $[a,b]\subset \R$ is not required to map into itself).
It was shown in \cite{bfmot13} that similar results hold for full plane
continua under a positively oriented branched covering map of the
plane. We will extend these results in this section. We will also provide,
in certain cases,  more information about the local behavior of the map
near the fixed point.

\subsection{Background from complex dynamics and continuum theory}
\label{ss:fxpt}
If $K$ is a
connected filled Julia set of a polynomial $P$, the Riemann map for $\C\sm K$ has also a dynamical
meaning as it conjugates $P|_{\C\sm K}$ and $z\mapsto z^d$ restricted to $\C\sm \cdisk$.
Call a $K$-ray \emph{rational} if has rational argument.

\begin{thm}[\cite{DH}]\label{t:ration}
Each rational $K$-ray lands; its landing point is a periodic or preperiodic regular point.
Conversely, each regular periodic point is the landing point of a non-empty finite collection of rational
$K$-rays.
\end{thm}

In \cite{GM}, the authors consider the partition of the plane by all invariant $K$-rays united with their landing points
(by definition and Lemma \ref{l:invar}, these landing points are exactly all non-rotational $P$-fixed points).
Denote the union of all invariant $K$-rays and their landing points by $\Sigma$.
Define a \emph{rotational object} as either a rotational fixed point or an invariant parabolic domain.

\begin{thm}[\cite{GM}]\label{t:gm1a}
Every component $A$ of $\C\sm\Sigma$ contains exactly one rotational object.
\end{thm}

In the next section, we prove a stronger result.

Theorem \ref{t:bb} is well-known (see, e.g., Corollary 5.9 of \cite{nad92}).

\begin{thm}[Boundary Bumping Theorem]
\label{t:bb} Let $X$ be a continuum, $V$ a proper nonempty open subset
of $X$, and $F$ any component of $V$. Then the closure of $F$ must
intersect the boundary of $V$.
\end{thm}

Let $Q\subset J$ be an invariant continuum and $x\in Q$ be a fixed point.
The next lemma follows from Theorem \ref{t:bb}.

\begin{lem}\label{l:xisin}
Let $F$ be a component of $Q\sm \{x\}$. Then $\ol{F}=F\cup \{x\}$ is a continuum.
\end{lem}

\begin{proof}
By Theorem \ref{t:bb}, the closure of $F$ contains $x$.
On the other hand, $\ol{F}\sm F$ cannot contain any other points.
\end{proof}

Lemma \ref{l:maxmod} follows from the fact that branched covering maps
of the plane are open maps. It is a topological version of the Maximum
Modulus Principle; we leave its proof to the reader.

\begin{lem}\label{l:maxmod}
If $f:\C\to \C$ is a branched covering map and $E$ is a continuum then
$f(\thu(E))=\thu(f(E))$. In particular, if $E$ is $f$-invariant then so
is $\thu(E)$.
\end{lem}

In the end of this section, we state the following important result.

\begin{thm}[Theorem 6.6 \cite{mcm94}]\label{t:66}
Suppose $x$ is a point in a continuum $Q$ such that $Q\sm \{x\}$ has at
least $n > 1$ connected components. Then there are $n$ external $Q$-rays that land
at $x$ and separate these components of $Q\sm \{x\}$. Thus,
at least $n$ external $Q$-rays land at $x$.
\end{thm}


\subsection{An overview of some fixed point theorems}\label{ss:cont-thm}
We need a number of definitions from \cite{bfmot13}, some of which are fairly standard.

Let $p:\R \to \uc$ be the covering map $p(x) = e^{2\pi i x}$.
Let $g:\uc\to \uc$ be a continuous map.
By the \emph{degree} $\deg(g)$ of $g$ we mean the number
$G(1)-G(0)$,
where $G: \R\to \R$ is a lift of the map $g$ to the universal covering space
$\R$ of $\uc$
(i.e., $p\circ G=g\circ p$). It is well-known that $\deg(g)$ is independent of the
choice of the lift.

Let $g:\uc\to \C$ be a continuous map and $f: g(\uc)\to \C$ be a fixed
point free continuous map. Define the map $v:\uc\to \uc$ by

$$v(t)=\frac{f(g(t)) - g(t)}{|f(g(t)) - g(t)|}.$$

\noindent Define the index $\ind(f, g)$ of $f$ with respect to $g$, by
$\ind(f, g)=\deg(v)$; the index $\ind(f, g)$ measures the net number of
revolutions of the vector $f(g(t)) - g(t)$ as $t$ travels through the
unit circle one revolution in the positive direction. If $S$ is a
Jordan curve and $f:S\to \C$ is a fixed point free map, we can define
$\ind(f, S)$ by using any positively oriented homeomorphic
parameterization $g$ of $S$ by $\uc$ as in this case the index does not
depend on $g$. If $x$ is a point, $S$ is the boundary of a Jordan disk
$D$ around $x$, and $F:D\to \C$ is a positively oriented homeomorphism
such that $F(x)=x$, and $F(S)$ is disjoint from $\ol{D}$,
then it is
easy to see that $\ind(F, S)=1$. If $x$ above is a unique fixed point
in $D$, then for all small Jordan disks $D'\subset D$ with fixed point
free boundaries $S'$ we have $\ind(F, S')=1$ or $0$ depending on
whether $x\in D'$ or $x\in D\sm\ol{D'}$. To see this note that in the first case
 we can isotope $S'$ to $S$ so that $x$ is always contained in the bounded complementary domain and, in the second case,
 we can isotope $S'$ to a point while always avoiding $x$. We talk about the \emph{local
index $\ind(F, x)$ of $F$ at $x$} so that $\ind(F, x)=1$ in the case
just described.

Let us now combine Definition 7.4.5 and Lemma 7.4.9 from \cite{bfmot13}
and define a type of fixed points to which our results will be applied.

\begin{dfn}\label{d:combi}
Suppose that $f$ is a positively oriented map of the plane to itself,
$X\subset \C$ is a full continuum 
and $p\in \bd(X)$ is a fixed point of $f$ such that:

\begin{enumerate}

\item there exists a neighborhood $U$ of $p$ such that $f|_U$
is one-to-one and $f(U\cap X)\subset X$,

\item there exists a ray $R\subset \C\sm X$ from infinity to $p$ such that $\ol{R} = R \cup {p}$,
and points of $R$ move away from $p$ along $R$;

\item there exists a nested sequence of closed disks $D_j \subset U$ with boundaries $S_j$
 such that $p\in D_j\sm S_j$ and
 $\bigcap D_j=\{p\}$ and $f(S_j\sm X)\cap D_j =\0$.

\end{enumerate}

Then we say that $f$ \emph{repels outside $X$ at $p$ in the narrow sense}.

\end{dfn}

We also need to define a class of non-invariant continua to which the results of \cite{bfmot13} apply.

\begin{dfn}\label{scracon}
Suppose that $f:\C\to\C$ is a positively oriented branched covering map and $X\subset \C$ is a full continuum.
Assume that there exist $n\ge 0$ disjoint full continua $Z_i$ such that the following properties hold:

\begin{enumerate}

\item $f(X)\sm X\subset \cup_i Z_i$;

\item for all $i$,  the intersection $Z_i\cap X=K_i$ is a full continuum;

\item for all $i$, either $K_i$ is a non-rotational weakly
    repelling fixed point or $f(K_i)\cap Z_i=\0$.
\end{enumerate}
Then we say that $f$  \emph{strongly scrambles the boundary
(of $X$)} and the continua $K_i$ are called \emph{exit continua (of $X$)}.
\end{dfn}

If $n=0$ in Definition~\ref{scracon}, then $X$ is invariant (i.e., $f(X)\subset X$).

\begin{rem}\label{zxgrow}
Since $Z_i$ and $Z_i\cap X=K_i\ne \0$ are full continua and
the sets $Z_i$ are pairwise disjoint, then $X\cup (\bigcup Z_i)$ is a
full continuum. Loosely, strongly  scrambling the boundary
means that $f(X)$ can only ``grow'' off $X$ \emph{within} the sets
$Z_i$ and \emph{through} the sets $K_i\subset X$ while either
$K_i$ is a non-rotational fixed point, or the image of the set
$K_i$ is disjoint from $Z_i$.
\end{rem}

The next theorem helps find certain fixed points in some continua.

\begin{thm}[Theorem 7.4.7 \cite{bfmot13}]\label{t:7.4.7}
Suppose that $f$ is a positively oriented
branched covering map of the plane with
only isolated fixed points, $X\subset \C$ is a full continuum,
and the following holds.

\begin{enumerate}

\item Each fixed point $p\in X$ belongs to the boundary of $X$, the index $\ind(f, p)$ at $p$ equals 1,
and $f$ repels outside $X$ at $p$ in the narrow sense.

\item The map $f$ strongly scrambles the boundary of $X$.
Moreover, in the notation from Definition \ref{scracon}, for each $i$, either $f(K_i)\cap Z_i = \0$,
or there exists a neighborhood $U_i$ of $K_i$ with $f(U_i\cap X)\subset X.$
\end{enumerate}
Then $X$ is a point.
\end{thm}

Recall that we consider only metrics in $\C$ that generate the same topology as the Euclidian metric.

\begin{lem}\label{l:repel-curves}
Let $f:\C\to\C$ be a positively oriented branched covering; let $x$ be
a weakly repelling fixed point of $f$ in the sense of a metric $d$. For
a small open $d$-disk $D_1$ of radius $\eps$ centered at $x$ with
boundary $S_1$ take iterated pullbacks of $D_1$ under the map $f$
choosing the branch of the inverse function that fixes $x$. Denote the
$n$-th pullback of $D_1$ by $D_n$, and set $\bd(D_n)=S_n$; then, for
every $n$, the set $D_n$ is a Jordan disk, and $\ol{D}_{n+1}\subset
D_n$. Moreover, $\{x\}=\cap D_i$.
\end{lem}

In the above case we say that the sequence $\{D_j\}$ is a \emph{repelling
sequence} for $x$.

\begin{proof}
We may assume that $\eps$ is small enough so that $d(f(y),x)>d(y,x)$
for all $y\in\ol D_1$ and $f$ is one-to-one on $\ol{D}_1$. Since $x$ is
weakly repelling and $x$ is fixed, $f(\bd(D_1))\cap D_1=\0$ and it is easy to see that
$\ol{D}_1\subset f(D_1)$
Since $f$ at $x$ is a local
homeomorphism, $\ol{D}_{n+1}\subset D_n$. We claim that $\{x\}=\cap
D_i$. Indeed let $\{x\}\subsetneqq \cap D_i$ (observe that $\cap
D_i=\cap \ol{D}_i$). Clearly, $f(\cap D_i)=\cap D_i$. On the other
hand, let $\delta=\max\{d(y, x):y\in \cap D_i\}$. Then $\cap D_i$ is
contained in the closed disk $E$ of radius $\delta$ centered at $x$,
and some points of $\cap D_i$ belong to $\bd(E)$. Since $x$ is weakly
repelling, $\bd(E)$ is mapped outside $E$ which makes $f(\cap D_i)=\cap
D_i$ impossible, a contradiction.
\end{proof}

\subsection{New fixed point results}
\label{ss:new-fxpt}
We return to the case of a complex polynomial $P$ with connected $K_P$.
The following construction allows to replace parabolic points with attracting ones in 
the topological category.

\begin{prop}
  \label{p:adjust}
There exists a topological branched covering $\widehat{P}$ of the sphere such that the following holds.
\begin{itemize}
  \item The maps $\widehat{P}$ and $P$ coincide outside of periodic parabolic domains of $P$.
  \item For a periodic parabolic domain $U$ of $P$ of period $n$, all periodic points of $\widehat{P}$ of period $n$ in $\bd(U)$ are weakly repelling.
  \item The map $\widehat{P}$ has a unique (super)attracting periodic point of period $n$ in any
  periodic parabolic domain $U$ of $P$ of period $n$,
  and $U$ is an immediate attracting basin for this point with respect to $\widehat{P}$.
\end{itemize}
\end{prop}

Observe that the action of $P$ and of $\hp$ 
on $K$-rays 
 and on $J$ is the same.
All parabolic points of $P$ are changed to repelling points of $\hp$
so that Theorem \ref{t:7.4.7} can be applied.

\begin{proof}
We will prove the proposition for the case when there is only one
periodic parabolic domain $U$ (in which case it must be invariant);
in the general case the arguments are similar.
Recall that by \cite{ry08} the boundary of $U$ is a Jordan curve.
Let $\psi:\ol\disk\to\ol U$ be a homeomorphism.
Consider the map $h:\uc\to \uc$ given by $h=\psi^{-1}\circ P\circ\psi$.

We claim that all fixed points of $h$ are weakly repelling.
Indeed, let $u$ be a fixed point of $h$; then $x=\psi(u)$ is a fixed point of $P$.
Let $P_x^{-1}$ be a local inverse of $P$ near $x$, and $W$ a repelling petal at $x$.
Then $P_x^{-1}(W)\subset W$, and $P_x^{-n}(W)$ converges to $x$ (here $P_x^{-n}$ is the $n$th iterate of $P_x^{-1}$).
Set $A=\psi^{-1}(\ol W\cap\bd(U))$.
Then $A$ contains a neighborhood of $u$ in $\uc$, we have $h_u^{-1}(A)\subset A$, and $h_u^{-n}(A)$ converges to $u$,
where $h_u^{-1}$ is a local inverse of $h$ near $u$, and $h_u^{-n}$ is the $n$th iterate of it.
It follows that $u$ is weakly repelling in $\uc$.

Consider the map $H:\ol\disk\to\ol\disk$ given by $w\mapsto |w|^2 h(w/|w|)$ for $w\ne 0$;
set $H(0)=0$ by continuity.
It is easy to see that all fixed points of $h$ in $\uc$ are weakly repelling for $H$.
Set $\widehat{P}=\psi\circ H\circ\psi^{-1}$ on $U$ and $\widehat{P}=P$ everywhere else.
\end{proof}

Theorem \ref{t:1a-new} strengthens Theorem \ref{t:gm1a}; recall that
by Theorem \ref{t:gm1a} there is a unique rotational object of $P$ in any component of $\C\sm \Sigma$.

\begin{thm}
  \label{t:1a-new}
Let $A$ be a component of $\C\sm \Sigma$, let $F\subset A\cup \fix$ be a non-degenerate continuum,
and let $P(F)\sm F$ be disjoint from $A$. Then the rotational object of $P$ in $A$ is contained in $\thu(F)$.
\end{thm}

\begin{proof}
  Apply Theorem \ref{t:7.4.7} to the map $f=\hp$ and the continuum $X=\thu(F)$.
Suppose that $X$ does not contain a fixed rotational point of $f$.
Then the assumptions of Theorem \ref{t:7.4.7} are satisfied.
The conclusion is that $X$ is a point, a contradiction.
Therefore, there is a fixed rotational point of $f$ in $X$.
This point is either a fixed rotational point of $P$, or the superattracting fixed point created in a parabolic domain of $P$.
In both cases, we obtain the desired.
\end{proof}

\section{Locally fully invariant continua}
\label{s:loc-max}
A continuum $Y\subset J_P$ is said to be \emph{locally fully invariant} (LFI) if $Y$ is a component of $P^{-1}(Y)$.
In this case, as is easy to see, $P(Y)=Y$, i.e., the set $Y$ is forward invariant.
Conceptually, LFI continua are related to the \emph{locally maximal} invariant sets (cf \cite{akk81} or even
\cite{sha66}).
Given an LFI set $Y$ let $\nu_Y(z)$ be the number of all $P$-preimages of $z$ in $Y$ counted with multiplicities.
The following lemma is a particular case of \cite[Lemma 4.1]{bopt16a} .

\begin{lem}[Lemma 4.1 \cite{bopt16a}]\label{l:4.1}
If $Y$ is locally fully invariant, then $\nu_Y$ is constant on $Y$.
\end{lem}

This constant is called the \emph{degree} of $Y$.

\begin{lem}
  \label{l:finv}
  Let $T\subset J_P$ be a continuum such that $P(T)=T$.
Then there is a minimal by inclusion LFI continuum $Y\supset T$.
Every periodic cutpoint of $Y$ is also a cutpoint of $T$.
\end{lem}

The LFI continuum $Y$ is called the \emph{LFI hull of $T$}.

\begin{proof}
Define a family of continua $T_\al$ indexed by ordinals $\al$ by transfinite induction.
Set $T_0=T$ and $T_{\al+1}$ to be the component of $P^{-1}(T_\al)$ 
containing $T$.
Finally, if $\al$ is a limit ordinal, let $T_\al$ be the closure of the union of all $T_\be$ with $\be<\al$.
Then $T_\be\subset T_\al$ whenever $\al<\be$.
It is clear that the transfinite sequence $T_\al$ stabilizes at an ordinal of cardinality at most that of the continuum.
That is, there is an ordinal $\al_0$ such that $T_\al=T_{\al_0}$ for all $\al>\al_0$.
Then $T_{\al_0}$ is the desired LFI continuum.

Now consider a periodic cutpoint $z$ of $Y$.
Replacing $P$ with a suitable iterate, we may assume that $P(z)=z$.
Let $W$ be the wedge formed by two adjacent external rays of $P$ landing at $z$ such that $T\cap W\ne\0$.
If $z$ is not a cutpoint of $T$, then $T\subset W$.
Then $T_\al\subset\ol W$ for all $\al$, a contradiction with $z$ being a cutpoint of $Y$.
\end{proof}

Let us study properties of LFI continua.

\begin{lem}
\label{l:lfi-1-1}
Let $Y$ be an LFI continuum for $P$.
If $P:Y\to Y$ is one-to-one, then there are no critical points of $P$ in $\thu(Y)$.
\end{lem}

\begin{proof}
  Suppose that $P:Y\to Y$ is one-to-one but there is a critical point $c$ of $P$ in $\thu(Y)$.
First, assume that a critical point $c$ of $P$ is in a bounded complementary component $U$ of $Y$.
Then $U$ is a Fatou component on which $P$ is not 1-to-1, hence,
by the classification of Fatou components, $P$ cannot be 1-1 on $\bd(U)\subset Y$, 
a contradiction.
Finally suppose that $c\in Y$.
Choose  $y_i\in Y$ so that $\lim y_i=P(c)$.
Since $Y$ is an LFI continuum, for $i$ large enough, multiple preimages of $y_i$ near $c$ must belong to $Y$, a contradiction.
\end{proof}

Recall the notion of an external map due to Douady and Hubbard \cite{DH-pl} in the polynomial-like setup and
 to Ha\"issinsky \cite{hai98} in a more general setup of \emph{polynomial silhouettes}.

\begin{dfn}[External map]
\label{d:extmap}
  Let $Y$ be a \emph{full} LFI continuum for $P$, and let $\phi:\disk\to\C^*\sm Y$ be a Riemann map.
Since $Y$ is LFI, there exists a neighborhood $U\supset Y$ such that no point of $U\sm Y$ can map into $Y$.
Hence we can choose $\eps>0$ so that the map $F=\phi^{-1}\circ P\circ \phi$ is defined and holomorphic on the
 annulus $A^-_\eps=\{z: 1-\eps<|z|<1\}$.
By the Schwarz reflection principle, extend $F$ to a holomorphic map of the annulus $A_\eps=\{z:1-\eps<|z|<(1-\eps)^{-1}\}$ to $\C$
 preserving $\uc$ and taking $A_\eps$ to another annulus around $\uc$.
The real analytic map $F:\uc\to\uc$ is the \emph{external map} of $Y$ with respect to $P$.
Since $F$ takes $A^-_\eps$ to a subset of the disk $|z|<1$, it has no critical points on $\uc$.
Abusing the terminology, we will sometimes call the analytic continuation of $F$ to $A_\eps$ also the external map.
\end{dfn}

In Lemma \ref{l:no-attr}, we let $Y$ denote an LFI continuum for $P$, and
 $F$ be the external map of $\thu(Y)$.

\begin{lem}
  \label{l:no-attr}
The external map $F:A_\eps\to\C$ has neither attracting periodic points in $\uc$
 nor parabolic periodic points in $\uc$ whose basins intersect $\uc$.
\end{lem}

\begin{proof}
  Assume the contrary: $z\in\uc$ is an attracting fixed point of $F^m$ or a parabolic point whose basin intersects $\uc$.
Then there is a disk $U$ such that $z\in\ol{U\cap\uc}$ and $F^m(U)\subset U$.
It follows that $P^m(V)\subset V$, where $V=\phi(U\cap\disk)$ and $\phi$ is as in Definition \ref{d:extmap}.
In particular, a circle arc containing $z$ is contained in $\ol U$. 
By the Denjoy-Wolff theorem (cf. Theorem 5.4 of \cite{M}), all $P^m$ orbits in $V$ converge to a $P^m$-fixed point $y\in Y$.
Then $y$ is parabolic or attracting, and $V\subset W$ where $W$ is a parabolic or attracting basin of $y$.
In both cases, the parabolic or attracting basins of $y$ are completely inside or completely outside $\thu(Y)$
(because $Y\subset J_P$).
Since 
points of $V$ do not belong to $\thu(Y)$ by construction, then $W\cap \thu(Y)=\0$.
Thus, $y$ is parabolic. The fact that $Y\subset J_P$ now implies that points of $Y$ close to $y$ are repelled from $y$ under $P^m$.
A contradiction with the fact that $\ol V\cap Y$ is mapped to $\ol V$ under all iterates of $P^m$.
\end{proof}

We can now deduce that $F$ is topologically expanding.

\begin{cor}
\label{c:zk}
Let $Y$ be an LFI continuum for $P$ of degree $k>1$.
Then the external map $F:\uc\to\uc$ of $\thu(Y)$ is topologically conjugate to the $k$-th tupling map $z\mapsto z^k$.
\end{cor}

\begin{proof}
By Lemma \ref{l:no-attr}, it follows from \cite{lyu89, bl89} (see also \cite{mms92}) that for every arc $I\subset \uc$
 there exists a number $n>0$ such that $F^n(I)\cap I\ne \0$.
This, in turn, implies (again together with Lemma \ref{l:no-attr}) that $F|_{\uc}$ is topologically conjugate to the map $z^k$ with appropriate $k$
 (see, e.g., \cite{mr07}).
\end{proof}

Let $Y$ be an LFI continuum for $P$.
Following Douady and Hubbard \cite{DH-pl}, say that $P$ is \emph{polynomial-like} (PL) near $\thu(Y)$
 if $P$ maps some neighborhood of $\thu(Y)$ onto a strictly bigger neighborhood as a branched covering.
In this case, points near of $\C\sm \thu(Y)$ near $\thu(Y)$ are repelled from $\thu(Y)$.

\begin{dfn}\label{d:outward}
Say an LFI set $Y$ contains an \emph{outward} parabolic periodic point (or cycle)
 if $Y$ contains a parabolic periodic point $x$ such that
 the corresponding parabolic Fatou domains (whose points are attracted to the orbit of $x$)
 are otherwise disjoint from $\thu(Y)$.
\end{dfn}

Theorem \ref{t:cmph} is, basically, a folklore result.


\begin{thm}[Theorem B \cite{bopt16a}]\label{t:cmph}
If $Y$ is an LFI continuum that contains no outward parabolic cycles, then $P$ is PL near $\thu(Y)$.
\end{thm}

Consider a sequence of sets $E_n$.
Say that $E_n$ \emph{accumulate on} $Y$ if, for every open neighborhood $U$ of $Y$, we have $E_n\subset U$ for large $n$
 (how large may depend on $U$).

\begin{prop}
  \label{p:no-acc}
  Consider an LFI set $Y$ for $P$ of degree $>1$.
  Let $x\in Y$ be a fixed point. 
  Let $E$ be a continuum such that sets $E_n=P^n(E)$ share $x$ with $Y$ and are otherwise disjoint from $Y$.
  If $x$ is not a cutpoint of $Y$, then $E_n$ cannot accumulate in $Y$.
\end{prop}

\begin{proof}
Assume the contrary: $E_n$ accumulate on $Y$.
Let $\phi$ be as in Definition \ref{d:extmap}.
The sets $\phi^{-1}(E_n\sm Y)$ accumulate in $\uc$.
Define $T\subset\uc$ as the set of points $z\in\uc$ such that $\phi^{-1}(E_n\sm Y)$ come arbitrarily close to $z$.
Clearly, $T$ is closed, connected and $F$-invariant, where $F:A_\eps\to\C$ is the external map for $\thu(Y)$.

If $T$ is a singleton, then $T$ attracts some nearby continua. Hence
$T$ is an attracting or parabolic fixed point of $F$. By
Lemma \ref{l:no-attr} it must be a parabolic fixed point of $F$,
and nearby continua must be contained in the corresponding parabolic
domain. However then all points close to these continua must also converge to
$T$ under $F$ while on the $P$-plane there are points close to $E_n\subset J_P$ that escape to infinity, a contradiction.
It follows by Corollary \ref{c:zk} that $T=\uc$.
Consider a pair of different regular (pre)periodic points $a$, $b$ in $Y$.
If there are no outward parabolic points in $Y$, then the existence of $a$ and $b$ follows from
 the Douady--Hubbard straightening theorem, since $P$ is PL near $Y$.
On the other hand, if there are outward parabolic points in $Y$, then choose $a$ and $b$ among the pullbacks of these points.
There are external $K_P$-rays $R_a$, $R_b$ landing at $a$ and $b$, respectively.

We claim that all $E_n\sm Y$ lie in one complementary component $W$ of $Y\cup R_a\cup R_b$.
This follows from the fact that $x$ is not a cutpoint of $Y$.
Thus $\phi^{-1}(E_n\sm Y)\subset\phi^{-1}(W)$.
However, $\phi^{-1}(W)$ is bounded by two topological rays to $\uc$ landing in $\uc$,
 and $\ol{\phi^{-1}(W)}\cap\uc$ is an arc that is a proper subset of $\uc$.
On the other hand, this arc must include $T=\uc$, a contradiction.
\end{proof}

\section{Invariant continua and their cutpoints}\label{s:inv-cont}
This section concludes the proof of the Main Theorem.

\medskip

\noindent{\textbf{More standing notation.}}
Assume that $Q\subset J$ is an invariant continuum, $x\in Q$ is a regular
fixed non-rotational point of $P$. Denote by $W$ the wedge between
invariant $K$-rays $R'$, $R''$ landing at $x$ so that there are no
$K$-rays in $W$ landing at $x$. Assume that the movement within $W$
from $R''$ to $R'$ is in the clockwise direction. Finally, assume that $Q\cap W\ne \0$.

\medskip


\label{s:IU_W}
We want to prove that $Q\cap W$ consists of just one component. 
The idea of the proof is based upon the dynamics of such components.
To study it we need a concept of a \emph{thread}.
Namely, a connected set $E\subset W$ such that $\ol E=E\cup\{x\}$ is called
a \emph{thread}.
A \emph{thread of $Q$} is a thread that is a subset of $Q$.
A thread of $Q$ is contained in a component of $Q\sm \{x\}$ contained in $W$;
by Lemma \ref{l:xisin}, each such component of $Q\sm \{x\}$ is a thread of $Q$.

If $E_1$ and $E_2$ are disjoint threads of $Q$, then by Theorem \ref{t:66} there is a
topological ray $T$ that separates $E_1$ from $E_2$ inside $W$; write
$E_1<E_2$ if $T$ separates $E_2$ from $R'$ in $W$.
Say that $E_1<E_2$ \emph{locally} if $E'_1<E'_2$ for some threads $E'_1\subset E_1$ and $E'_2\subset E_2$.

If $E$ is a thread of $Q$ containing no preimages of $x$, then its
$P$-image $P(E)$ is a thread too (e.g., this is the case if $E$ is
sufficiently small). The set $P(E)$ may also contain smaller threads.
Let us study their location with respect to $E$.
Recall that threads are by definition subsets of $W$ even though $W$ is omitted from the notation and terminology.
Also, recall that a \emph{topological ray to $Q$} is a topological ray disjoint from $Q$ but accumulating in $Q$
 (as the parameter on the ray tends to $0$).

\begin{lem} \label{l:ndthr}
Let $E'$ and $E''$ be threads of $Q$  such that $E'<E''$ locally but $E'\cap E''\ne\0$.
There are no threads $F$ of $Q$ disjoint from $E'\cup E''$ and such that $E'<F<E''$.
\end{lem}

\begin{proof}
Let $V$ be the complementary component of $\ol E'\cup\ol E''$ containing $F$.
If $V$ is unbounded, then there exists a $E'\cup E''\cup F$-ray $T$ from a point
of $F$ to infinity. Evidently, $F\cup T$ separate $E'$ from $E''$ in $W$, a contradiction with $E'\cap E''\ne\0$.
Thus $V$ is bounded. Since no point $x\in J$ belongs to $V$,  we have a contradiction with
$F\subset V$.
\end{proof}

The following lemma allows to find a ``monotone'' infinite sequence of disjoint threads.

\begin{lem} \label{l:Eseq}
Let $E<F$ be disjoint threads of $Q$ such that $x\notin P(E)$ and $F\subset P(E)$.
Then there exists a thread $T\subset E$ such that all its images
$T<P(T)<\dots$ are pairwise disjoint threads that never contain a preimage of $x$.
\end{lem}

\begin{proof}
Construct a null sequence of threads $E\supset E^1\supset \dots$. For each $i$, choose
$n_i$ such that $E^i,$ $P(E^i),$ $\dots,$ $P^{n_i-1}(E^i)$ are all pairwise disjoint threads not containing
immediate preimages of $x$ while $P^{n_i}(E^i)$ contains an immediate preimage of $x$
(if for some $j$ no image of $E^j$ contains an immediate preimage of $x$, set $T=E^j$).
By continuity and construction, $n_i\nearrow \infty$ as $i\to \infty$.
Let $y\in P^{-1}(x)\cap W$. Suppose that $y\in P^{n_i}(E^i)$. Then by Lemma \ref{l:ndthr} we may assume that
for $j>i$ no image of $E^j$ contains $y$. Hence if $j$ is sufficiently large, immediate preimages of
$x$ do not belong to images of $E^j$. Set $T=E^j$.
\end{proof}

Lemma \ref{l:enowander} describes the dynamics of threads.

\begin{lem}\label{l:enowander}
Suppose that $E$ is a thread of $Q$.
Then its image $P(E)$ cannot contain a thread disjoint from $E$.
\end{lem}

\begin{proof}
If there is a thread $E_1$ contained in $P(E)$ and disjoint from $E$, and
$E_0$ is a pullback of $E_1$ that is a thread in $E$, then $E_0\cap E_1=\0$.
Assume that $E_0<E_1$ and consider possible location
of sets $E_n=P^n(E_0)$; call them \emph{$E$-sets} for brevity.
By Lemma \ref{l:Eseq}, we may assume that all $E_n$ are disjoint threads of $Q$ not containing preimages of $x$.

Let $T$ be the \emph{topological upper limit} of the sequence of sets $E_n$, i.e., the set of
all limit points of sequences of points $z_i\in E_i, i=1, 2, \dots$.
Clearly, $T$ is closed, $P(T)=T$, and $x\in T$.
Note that $T$ coincides with the union of all Hausdorff limits of sequences of sets $E_{i_n}$ with
$i_n\nearrow \infty$; also, all $E$-sets are disjoint from $T$
by Lemma \ref{l:ndthr}. Let $Y$ be the LFI hull of $T$.
Then $Y$ is disjoint from the $E$-sets because
if, say, $E_n\cap Y\ne\0$,
then, again by Lemma \ref{l:ndthr}, the sets $E_m$ cannot be contained in $J$ for $m>n$.

By Proposition \ref{p:no-acc}, the map $P:Y\to Y$ is one-to-one.
Hence $P:T\to T$ is one-to-one.
By Lemma \ref{l:lfi-1-1}, there are no critical points of $P$ in $\thu(Y)$.
In particular, $\thu(Y)$ cannot include any attracting or parabolic domains.
By Theorem \ref{t:1a-new}, there is a $P$-fixed rotational point $z\in\thu(Y)$.
If $z$ is a cutpoint of $T$, then all $E_n$ with large $n$ are in the same wedge bounded by external $P$-rays landing at $z$.
A contradiction.
Hence $z$ is Siegel or Cremer.
Let $O$ be a neighborhood of $\thu(Y)$ such that $\ol O$ contains no critical points.
Then the set of all points whose forward orbits stay forever in $\ol{O}$ contains a component $\tilde Y\supset Y$.
By \cite{per94, per97}, the set $\tilde Y$ is a so-called \emph{hedgehog} of $z$.
It cannot contain fixed points other than $z$, a contradiction with $x\in Y$.
\end{proof}

This implies the following corollary.

\begin{cor}\label{c:all-inv}
If $F$ is a component of $Q\cap W$ then all small threads of $F$ map back to $F$.
Thus, if $F$ contains no immediate preimages of $x$ then $P(F)\subset F$. In particular,
there are at most finitely many components of $Q\cap W$.
\end{cor}

\begin{proof}
The first claim is immediate. It implies the second one. Since there are at most
finitely many immediate  preimages of $x$ in $W$, the last claim follows.
\end{proof}

We can now prove Lemma \ref{l:747ok} which allows us to use Theorem \ref{t:7.4.7}.

\begin{lem}\label{l:747ok}
If $F$ is a component of $Q\cap W$, then all points of $F$ sufficiently close to $x$ map to $F$.
\end{lem}

\begin{proof}
Suppose that $y\in F$ maps to a component $F'\ne F$ of $Q\cap W$.
By Corollary \ref{c:all-inv}, there are preimages $x'$ of $x$ that belong to $F$.
Let $X'=P^{-1}(x)\cap F$; clearly, $X'$ is finite.
Choose a component $F_y$ of  $F\sm X'$ that contains $y$.
Then $P(F_y)\subset F'$. Hence $P(\ol{F_y})\subset F'\cup x$.
Now, by way of contradiction suppose that there is a sequence $y_i\in F$ converging to $x$ such that $P(y_i)\notin F$.
Passing to a subsequence and using that, by Corollary \ref{c:all-inv}, there are only finitely many components of $Q\cap W$,
we may assume that for some component $F'$ of $Q\cap W$ and for some point $x'\in X'$ the sets $\ol{F}_{y_i}$
converge to a continuum $B\subset \ol{F}$ containing $x$ and $x'$ such that $P(B)\subset \ol{F'}=F'\cup \{x\}$.
However this contradicts Lemma \ref{l:enowander}. Thus, all points of $F$ sufficiently close to $x$ map to $F$.
\end{proof}

We are ready to prove Theorem \ref{t:only1}.

\begin{thm}\label{t:only1}
The set $Q\cap W$ is connected.
\end{thm}

\begin{proof}
Let $F$ be a component of $Q\cap W$. Denote the union of all invariant $K$-rays and their landing points by $\Sigma$.
Let $A$ be the component of $\C\sm \Sigma$ that contains points of $F$ close to $x$.
By Theorem \ref{t:1a-new}, the set
$\thu(F)$ contains the rotational object of $P$ in $A$.
Since this applies to any component of $Q\cap W$, we conclude that there is just one component, i.e., that $F=Q\cap W$.
\end{proof}

Evidently, Theorem \ref{t:only1} implies Main Theorem.


\begin{thebibliography}{9999}

\bibitem[AKK81]{akk81} V. M. Alekseev, A. B. Katok, A. G. Kushnirenko, \emph{Three
papers on Dynamical Systems}, AMS Traslations (Series 2), American Mathematical Society,
Providence, RI (1981)




\bibitem[BFMOT13]{bfmot13} A. Blokh, R. Fokkink, J. Mayer, L. Oversteegen, E.
Tymchatyn, \emph{Fixed point theorems for plane continua with
applications}, Memoirs of the American Mathematical Society,
\textbf{224} (2013), no. 1053.

\bibitem[BL89]{bl89} A. Blokh, M. Lyubich, \emph{Nonexistence of wandering intervals
and structure of topological attractors of one-dimensional dynamical systems. II. The smooth case.},
Ergodic Theory Dynam. Systems \textbf{9} (1989), 751--758.

\bibitem[BOPT16a]{bopt16a} A. Blokh, L. Oversteegen, R. Ptacek, V.
    Timorin, \emph{Quadratic-like dynamics of cubic polynomials},
    Comm. Math. Phys. \textbf{341} (2016), 733--749.

\bibitem[BOT20]{bot20}  A. Blokh, L. Oversteegen, V.
    Timorin, \emph{On critical renormalization of complex polynomials}, arXiv:2008.06689 (2020)

\bibitem[BD88]{bd88}
B. Branner, A. Douady,
\emph{Surgery on complex polynomials}, in: Holomorphic Dynamics (Proceedings of the Second International
Colloquium on Dynamical Systems, held in Mexico, July 1986), Eds: X. Gomez-Mont, J. Seade, A. Verjovski (1988), 11--72.

\bibitem[BF14]{bf14}
B. Branner, N. Fagella,
\emph{Quasiconformal surgery in holomorphic dynamics}, Cambridge Studies in Advanced Mathematics, vol. 141.
Cambridge University Press, Cambridge, 2014.

\bibitem[Bro11]{bro11} L. E. J. Brouwer, \emph{Ueber Abbildungen von Mannigfaltigkeiten},
Mathematische Annalen \textbf{71} (1911), 97--115.

\bibitem[DH8485]{DH}
A. Douady, J.H. Hubbard, \emph{\'Etude dynamique des polyn\^omes complex I \& II}
Publ. Math. Orsay (1984--85).

\bibitem[DH85]{DH-pl} A. Douady, J.H. Hubbard, \emph{On the dynamics of
    polynomial-like mappings},  Ann. Sci. \'Ecole Norm. Sup. (4)
    \textbf{18} (1985), no. 2, 287--343.

\bibitem[GM93]{GM} L.R. Goldberg, J. Milnor, \emph{Fixed points of
    polynomial maps. Part II. Fixed point portraits},
    Ann. Sci. \'Ecole Norm. Sup. (4) \textbf{26}, no. 1
    (1993), 51--98

\bibitem[Ha\"i98]{hai98}
P. Ha\"issinsky, \emph{Applications de la chirurgie holomorphe notamment aux points paraboliques},
PhD Thesis, Universit\'e Paris Sud, 1998.

\bibitem[HK95]{hk95} B. Hasselblatt, A. Katok, Introduction to the modern theory of dynamical systems,
Encyclopedia of Mathematics and its Applications, vol. 54, Cambridge University Press, Cambridge (1995)


\bibitem[LP96]{lepr} G. Levin, F. Przytycki, \emph{External rays to
    periodic points,} Isr. J. Math. \textbf{94} (1996), pp. 29--57.

\bibitem[Lom15]{lom15}
L. Lomonaco, \emph{Parabolic-like mappings}, Erg. Th. and Dyn. Syst.,
\textbf{35}, no. 7 (2015), 2171--2197.

\bibitem[Lyu89]{lyu89} M. Lyubich, \emph{Nonexistence of wandering intervals and structure of topological attractors of
 one-dimensional dynamical systems. I. The case of negative Schwarzian derivative.},
 Ergodic Theory Dynam. Systems \textbf{9} (1989), 737--749.

\bibitem[MMS92]{mms92} M. Martens, W. de Melo, S. van Strien, \emph{Julia-Fatou-Sullivan theory for real one-dimensional dynamics,}
Acta Math. \textbf{168} (1992), 273--318.

\bibitem[McM94]{mcm94} C. McMullen, \emph{Complex dynamics and renormalization},
Annals of Mathematics Studies \textbf{135}, Princeton University Press, Princeton, NJ
(1994)

\bibitem[Mil06]{M} J. Milnor, ``Dynamics in one Complex Variable'', 3rd
    ed., Princeton University Press, Princeton (2006), ~viii+304pp.

\bibitem[MR07]{mr07} M. Misiurewicz, A. Rodrigues, \emph{Double standard maps},
Comm. Math. Phys. \textbf{273} (2007), 37--65.

\bibitem[Nad92]{nad92} S. Nadler, ``Continuum theory: an
    introduction'', Monographs and Textbooks in Pure and Applied Mathematics, vol. 158.
    Marcel Dekker, Inc., New York, 1992.

\bibitem[Per94]{per94} R. Perez-Marco, \emph{Topology of Julia sets and
      hedgehogs}, Publications Math\'ematiques d'Orsay \textbf{94-48} (1994).

\bibitem[Per97]{per97} R. Perez-Marco, \emph{Fixed points and circle
      maps}, Acta Math. \textbf{179} (1997), pp. 243--294.

\bibitem[RY08]{ry08} P. Roesch, Y. Yin, \emph{The boundary of bounded
    polynomial Fatou components},  Comptes Rendus Mathematique \textbf{346},
    877--880.

\bibitem[Sha66]{sha66} A. N. Sharkovsky, \emph{Partially ordered system of attracting sets} (Russian)
Dokl. Akad. Nauk SSSR, \textbf{170} (1966), 1276-–1278.

\bibitem[Ste35]{ste35} Sternbach, Problem 107 (1935), in: The Scottish Book: Mathematics from the Scottish
Caf\'e, Birkhauser, Boston (1981).

\bibitem[Sul85]{sul85} D. Sullivan, \emph{Quasiconformal homeomorphisms and dynamics. I.
Solution of the Fatou-Julia problem on wandering domains}, Ann. of
Math. (2), \textbf{122} (1985), no. 3, 401--418.

\end{thebibliography}
\end{document}